\documentclass[12pt]{amsart}

\newtheorem{theorem}{Theorem}[section]
\newtheorem{proposition}[theorem]{Proposition}
\newtheorem{lemma}[theorem]{Lemma}

\newtheorem{remark}{Remark}[section]

\newtheorem{definition}{Definition}[section]
\newtheorem{maintheorem}{Theorem}

\newcommand{\blanksquare}{\,\,\,$\sqcup\!\!\!\!\sqcap$}

\newcounter{example}
{\par}

\usepackage{amscd}
\usepackage{threeparttable}
\usepackage{amssymb}
\usepackage{amsfonts}
\usepackage{epsfig}
\usepackage{graphicx}
\usepackage{amsmath}
\usepackage{color}
\usepackage[all, 2cell]{xy} \UseAllTwocells \SilentMatrices
\usepackage{indentfirst}

\begin{document}

\title[Shadowing, expansiveness and stability]{Shadowing, expansiveness and stability of divergence-free vector fields}

\author[C\'elia Ferreira]{C\'elia Ferreira}
\address{Departamento de Matem\'atica, Universidade do Porto, 
Rua do Cam\-po Alegre, 687, 
4169-007 Porto, Portugal}
\email{celiam@fc.up.pt}

\begin{abstract}
Let $X$ be a divergence-free vector field defined on a closed, connected Riemannian manifold. 
In this paper, we show the equivalence between the following conditions:
\begin{itemize}
	\item $X$ is in the $C^1$-interior of the set of \emph{expansive} divergence-free vector fields.
	\item $X$ is in the $C^1$-interior of the set of divergence-free vector fields which satisfy the \emph{shadowing property}.
	\item $X$ is in the $C^1$-interior of the set of divergence-free vector fields which satisfy the \emph{Lipschitz shadowing property}.
	\item $X$ has no singularities and $X$ is \emph{Anosov}.
\end{itemize}   
\end{abstract}

\maketitle        
          

\begin{section}{Introduction and statement of the results}\label{resultstat}

Let $M$ be an $n$-dimensional, $n\geq 3$, closed, connected and smooth Riemannian manifold, endowed with a volume form, which has associated a measure $\mu$, called the Lebesgue measure, and let $d$ denote the Riemannian distance. 
Let $\mathfrak{X}^s(M)$ be the set of vector fields and let $\mathfrak{X}_{\mu}^s(M)$ be the set of divergence-free vector fields, both defined on $M$ and endowed with the $C^s$ Whitney topology, $s\geq 1$. 
From now on, we consider $s=1$.
A vector field $X$ has associated a flow, denoted by $X^t$, $t\in\mathbb{R}$.
Denote by $Per(X)$ the union of the \emph{closed orbits} of $X$ and by $Sing(X)$ the union of the \emph{singularities} of $X$. A subset of $M$ is said to be \emph{regular} if it has no singularities. 
Denote by $Crit(X)$ the set of the closed orbits and the singularities associated to $X$.
A singularity $p$ is \emph{linear} if there exist smooth local coordinates around $p$ such that $X$ is linear and equal to $DX(p)$ in these coordinates (see \cite[Definition 4.1]{V1}).

Take a $C^1$-vector field and a regular point $x$ in $M$ and let \linebreak$N_x:=X(x)^{\bot}\subset T_xM$ denote the ($\dim(M)-1$)-di\-men\-sion\-al normal bundle of $X$ at $x$ and $N_{x,r}=N_x\cap\{u\in T_xM:\|u\|<r\}$, for $r>0$. Since, in general, $N_x$ is not $DX^t_x$-invariant, we define the \emph{linear Poincar\'{e} flow} $$P^t_X(x):=\Pi_{X^t(x)}\circ DX^t_x,$$ where $\Pi_{X^t(x)}: T_{X^t(x)}M\rightarrow N_{X^t(x)}$ is the canonical orthogonal projection.

Let $\Lambda$ be a compact, $X^t$-invariant and regular set. 
If $N_{\Lambda}$ admits a $P_X^t$-invariant splitting $N_{\Lambda}=N_{\Lambda}^s\oplus N_{\Lambda}^u$, such that there is ${\ell}>0$ satisfying 
\begin{center}
$\|P_X^{\ell}(x)|_{N^s_x}\|\leq \dfrac{1}{2}$ and $\|P_X^{-{\ell}}(X^{\ell}(x))|_{N^u_{X^{\ell}(x)}}\|\leq\dfrac{1}{2}$, 
\end{center}
for any $x\in\Lambda$, we say that $\Lambda$ is \emph{hyperbolic}.
A vector field $X$ is said to be \emph{Anosov} if the whole manifold $M$ is hyperbolic. Let $\mathcal{A}^1_{\mu}(M)$ denote the set of Anosov $C^1$-divergence-free vector fields.

Take $T>0$ and $\delta>0$. A map $\psi:\mathbb{R}\rightarrow M$ is a \emph{$(\delta,T)$-pseudo-orbit} of a flow $X^t$ if, for any $\tau\in\mathbb{R}$, $d(X^t(\psi(\tau)),\psi(\tau+t))<\delta$, for any $\left|t\right|\leq T$.

Take $\epsilon>0$. A pseudo-orbit $\psi$ of a flow $X^t$ is \emph{$\epsilon$-shadowed} by some orbit of $X^t$ if there is $x\in M$ and an increasing homeomorphism \linebreak$\alpha:\mathbb{R}\rightarrow\mathbb{R}$, called reparametrization, which satisfies $\alpha(0)={0}$ and such that $d(X^{\alpha(t)}(x),\psi(t))<\epsilon$, for every $t\in\mathbb{R}$.

\begin{definition}
A $C^1$-vector field $X$ satisfies the \emph{shadowing property} if for any $\epsilon>0$ there is $\delta>0$ such that any $(\delta,T)$-pseudo-orbit $\psi$, for $T>0$, is $\epsilon$-shadowed by some orbit of $X$.
Let $\mathcal{S}^1(M)$ and $\mathcal{S}_{\mu}^1(M)$ denote the sets of vector fields in $\mathfrak{X}^1(M)$ and $\mathfrak{X}_{\mu}^1(M)$, respectively, satisfying the shadowing property.
\end{definition}

In the mid 1990's (see \cite{S}) it was shown that a dissipative diffeomorphism in the $C^1$-interior of the set of diffeomorphisms with the shadowing property is structurally stable. More recently, Lee and Sakai (see \cite{LS}) proved that if $X\in int \:\mathcal{S}^1(M)$ and has no singularities then $X$ satisfies the Axiom A and the strong transversality conditions, where $int S$ stands for the $C^1$-interior of a set $S\subset\mathfrak{X}^1(M)$.

Now, we introduce a weaker definition.

\begin{definition}
A $C^1$-vector field $X$ satisfies the \emph{Lipschitz shadowing property} if there are positive constants $\ell$ and $\delta_0$ such that any $(\delta,T)$-pseudo-orbit $\psi$, with $T>0$ and $\delta\leq\delta_0$ is $\ell\delta$-shadowed by an orbit of $X$.
Let $\mathcal{LS}^1(M)$ and $\mathcal{LS}_{\mu}^1(M)$ denote the sets of vector fields in $\mathfrak{X}^1(M)$ and $\mathfrak{X}_{\mu}^1(M)$, respectively, satisfying the Lipschitz shadowing property.
\end{definition}

It is immediate, from the previous definitions, that $\mathcal{LS}^1(M)\subset\mathcal{S}^1(M)$ and $\mathcal{LS}_{\mu}^1(M)\subset\mathcal{S}_{\mu}^1(M)$.
In \cite{T}, Tikhomirov proved that, for dissipative vector fields, Lipschitz shadowing is equivalent to structural stability.
Recently, Pilyugin and Tikhomirov proved the same result for dissi\-pa\-ti\-ve diffeomorphisms (see \cite{PT}).
We can find in \cite{P} the proof of that Anosov vector fields satisfy the Lipschitz shadowing property.

Let us now present the notion of \emph{expansive vector field}.
\begin{definition}
A $C^1$-vector field $X$ is \emph{expansive} 
if for any $\epsilon>0$ there is $\delta>0$ such that if $d(X^t(x),X^{\alpha(t)}(y))\leq\delta$, 
for all $t\in\mathbb{R}$, for $x,y\in M$ and a continuous map $\alpha:\mathbb{R}\rightarrow\mathbb{R}$
with $\alpha(0)=0$, then $y=X^{s}(x)$, where $\left|s\right|\leq\epsilon$. 
Denote by $\mathcal{E}^1(M)\subset\mathfrak{X}^1(M)$ the set of \emph{expansive vector fields} and by $\mathcal{E}_{\mu}^1(M)\subset\mathfrak{X}^1_{\mu}(M)$ the set of \emph{divergence-free expansive vector fields}, both endowed with the $C^1$ Whitney topology.
\end{definition}

In 1970's, Ma\~{n}\'{e} proved that if a dissipative diffeomorphism $f$ is in the $C^1$-interior of the set of expansive diffeomorphisms then $f$ is Axiom A and satisfies the quasi-transversality condition (see \cite{Ma}). 
Later, in \cite{MSS}, Moriyasu, Sakai and Sun proved the same result for dissipative vector fields. 
Moreover, they proved that if $X\in int\:\mathcal{E}^1(M)$ and has the shadowing property then $X$ is Anosov. 
Recently, Pilyugin and Tikhomirov proved that an expansive dissipative diffeomorphism having the Lipschitz shadowing property is Anosov (see \cite{PT}).

In this article, we intend to characterize divergence-free vector fields, with a topological property of Anosov systems, such as topological stability under $C^1$-open conditions: shadowing and expansiveness. We prove the following:

\begin{maintheorem}\label{mainth1}
For the divergence-free setting, one has that
$$int \:\mathcal{E}_{\mu}^1(M)=int \:\mathcal{S}_{\mu}^1(M)=int \:\mathcal{LS}_{\mu}^1(M)=\mathcal{A}^1_{\mu}(M).$$
\end{maintheorem}

\end{section}


\begin{section}{Definitions and auxiliary results}

In this section, we state some definitions and present some results that will be used in the proofs.

Let $\Lambda$ be a compact, $X^t$-invariant and regular set. Consider a splitting $N=N^1\oplus\cdots\oplus N^{k}$ over $\Lambda$, for $1\leq k\leq n-1$, such that all the subbundles have constant dimension. This splitting is \emph{dominated} if it is $P_X^t$-invariant and there exists ${\ell}>0$ such that, for every $0\leq i<j\leq k$ and every $x\in\Lambda$, one has
\begin{center}
${\|P_X^{\ell}(x)|_{{{N}}_x^i}\|}\cdot{\|P_X^{-\ell}(X^{\ell}(x))|_{{{N}}_{X^{\ell}(x)}^j}\|}\leq\dfrac{1}{2}, \: \: \forall \: x\in \Lambda.$
\end{center}

The following result can be obtained following the ideas presented in \cite[Proposition 2.4]{BR}.

\begin{theorem}\label{BGV}
Let $X\in\mathfrak{X}^1_{\mu}(M)$ and let $\mathcal{U}$ be a small $C^1$-neighbourhood of $X$. Then, for any $\epsilon>0$, there exist $l,\tau>0$ such that, for any $Y\in \mathcal{U}$ and any closed orbit $x$ of $Y^t$ of period $\pi(x)>\tau$,
\begin{itemize}
	\item either ${P}_Y^t$ admits an $l$-dominated splitting over the $Y^t$-orbit of $x$
  \item or else for any neighbourhood $U$ of $x$, there exists an $\epsilon$-$C^1$-per\-tur\-ba\-tion $\tilde{Y}$ of $Y$, coinciding with $Y$ outside $U$ and along the orbit of $x$, such that ${P}^{\pi{(x)}}_{\tilde{Y}}(x)=id$, where $id$ denotes the identity on $N_x$.
\end{itemize}
\end{theorem}

To prove Theorem \ref{mainth1}, we also need to state the definition of \emph{star vector field}.

\begin{definition}\label{starflowdef}
A $C^1$-vector field $X$ is a \emph{star vector field} if there exists a $C^1$-neighborhood $\mathcal{U}$ of $X$ in  $\mathfrak{X}^1(M)$ such that if $Y\in\mathcal{U}$ then every point in $Crit(Y)$ is hyperbolic. 
Moreover, a $C^1$-divergence-free vector field $X$ is a \emph{{divergence-free star vector field}} if there exists a $C^1$-neighborhood $\mathcal{U}$ of $X$ in $\mathfrak{X}_{\mu}^1(M)$ such that if $Y\in\mathcal{U}$ then every point in $Crit(Y)$ is hyperbolic. 
The set of star vector fields is denoted by $\mathcal{G}^1(M)$ and the set of divergence-free star vector fields is denoted by $\mathcal{G}_{\mu}^1(M)$.
\end{definition}

Accordingly with this definition, in \cite[Theorem 1]{F} it is proved the following result.

\begin{theorem}\label{mainth}
If $X\in\mathcal{G}^1_{\mu}(M)$ then $Sing(X)=\emptyset$ and $X$ is Anosov.
\end{theorem}

A $3$-dimensional proof of this result is presented in \cite{BR2} and a version for $4$-dimensional symplectic Hamiltonian vector fields can be found in \cite{BFR}.

The following result says that the linear Poincar\'e flow cannot admit a dominated splitting over the set of regular points of $M$ if the vector field has a linear hyperbolic singularity of saddle-type.

\begin{proposition}\cite[Proposition 4.1]{V1}\label{vivier}
If $X\in\mathfrak{X}^1(M)$ admits a linear hyperbolic singularity of saddle-type then ${P}_X^t$ does not admit any dominated splitting over $M\backslash Sing(X)$.
\end{proposition}

A vector field $X$ is \emph{topologically mixing} if, given any nonempty open sets $U,V\subset M$, there is $T>0$ such that, for any $t\geq T$, we have $X^t(U)\cap V\neq\emptyset$.
We end this section with a result stating that, $C^1$-generically, the divergence-free vector fields are topologically mixing.

\begin{theorem}\cite[Theorem 1.1]{B1}\label{topmix}
There exists a $C^1$-residual subset \linebreak$\mathcal{R}\subset\mathfrak{X}^1_{\mu}(M)$ such that if $X\in\mathcal{R}$ then $X$ is a topologically mixing vector field.
\end{theorem} 

\end{section}


\begin{section}{Proof of the theorem}

\begin{lemma}\label{mainlemma}
If $X\in int \: \mathcal{E}_{\mu}^1(M)$ then any closed orbit of $X$ is hyperbolic.
\end{lemma}

\begin{proof}

Take $X\in int \: \mathcal{E}_{\mu}^1(M)$ and $\mathcal{U}$ a $C^1$-neighbourhood of $X$ in $\mathcal{E}_{\mu}^1(M)$. 
Let $p$ be a point in a closed orbit of $X$ with period $\pi>0$ and $U_p$ a small neighbourhood of $p$ in $M$. 
By contradiction, assume that there is an eigenvalue $\lambda$ of $P_X^{\pi}(p)$ such that $\left|\lambda\right|=1$.

Applying Zuppa's Theorem (see \cite{Z}), we can find $Y\in\mathcal{U}$ such that $Y\in\mathfrak{X}^{\infty}_{\mu}(M)$, $Y^{\pi}(p)=p$ and $P_Y^{\pi}(p)$ has an eigenvalue $\lambda$ with $\left|\lambda\right|=1$.
\begin{remark}
Notice that if $P_Y^{\pi}(p)$ has not an eigenvalue $\lambda$ with $\left|\lambda\right|=1$, it has an eigenvalue $\tilde{\lambda}$ such that $|\tilde{\lambda}|\approx1$. So, we just have to perform a $C^1$-conservative perturbation $Z$ of $Y$, by \cite[Lemma 3.2]{BR}, such that $P_Z^{\pi}(p)$ has an eigenvalue $\bar{\lambda}$ with $|\bar{\lambda}|=1$. 
\end{remark}
Accordingly with Moser's Theorem (see \cite{M}), there is a smooth conserva\-ti\-ve change of coordinates $\varphi_p:U_p\rightarrow T_{p}M$ such that $\varphi_p(p)=\vec{0}$. 
Let $f_Y: \varphi^{-1}_p(N_{p})\rightarrow \Sigma$ be the Poincar\'e map associated to $Y^t$, where $\Sigma$ denotes the Poincar\'e section through $p$, and take $\mathcal{V}$ a $C^1$-neighbourhood of $f_Y$. 

By \cite[Lemma 3.2]{BR}, taking $\mathcal{T}$ a small flowbox of $Y^{[0,t_0]}(p)$, \linebreak$0<t_0<\pi$, we have that there are $Z\in\mathcal{U}$, $f_Z\in\mathcal{V}$ and $\epsilon>0$ such that: \linebreak$Z^t(p)=Y^t(p)$, $t\in\mathbb{R}$; $P_Z^{t_0}(p)=P_Y^{t_0}(p)$; $Z|_{\mathcal{T}^c}=Y|_{\mathcal{T}^c}$ and
\begin{equation}
f_Z(x) = \left\{ 
\begin{array}{ll}
\varphi_p^{-1}\circ P^{\pi}_Y(p)\circ\varphi_p(x) & \textrm{, $x\in B_{\epsilon/4}(p)\cap\varphi^{-1}_p(N_{p})$}\nonumber\\
f_Y(x) & \textrm{, $x\notin B_{\epsilon}(p)\cap\varphi^{-1}_p(N_{p})$}.
\end{array} 
\right.
\end{equation}

Notice that $P^{\pi}_Z(p)$ still has an eigenvalue $\lambda$ with $\left|\lambda\right|=1$.

Since $Z\in\mathcal{E}^1_{\mu}(M)$, for a sufficiently small $\epsilon>0$, there is $0<\delta<\epsilon$ such that if $d(Z^t(x),Z^{\alpha(t)}(y))\leq\delta$, for any $t\in\mathbb{R}$, $x,y\in M$ and $\alpha:\mathbb{R}\rightarrow\mathbb{R}$ continuous such that $\alpha(0)=0$, then $y=Z^s(x)$, where $\left|s\right|\leq\epsilon$.

Take $0<\delta'<\delta$ such that if $x,y\in M$ satisfy $d(x,y)<\delta'$ then $d(Z^t(x),Z^t(y))<\delta$, for $0\leq t\leq\pi$.

Firstly, assume that $\lambda=1$ and fix the associated non-zero eigenvector $v$ such that $\|v\|<\delta'$. Take $\varphi^{-1}_p(v)\in\varphi^{-1}_p(N_{p})\backslash\{p\}$ and note that 
$$f_Z(\varphi^{-1}_p(v))=\varphi^{-1}_p\circ P^{\pi}_Y(p)\circ\varphi_p(\varphi^{-1}_p(v))=\varphi^{-1}_p\circ P^{\pi}_Y(p)(v)=\varphi^{-1}_p(v).$$
So, $d(p,\varphi^{-1}_p(v))=d(p,f_Z(\varphi^{-1}_p(v)))=\|v\|<\delta'$. 
Then, as was mentioned before, $d(Z^t(p),Z^t(\varphi^{-1}_p(v)))<\delta$, for $0\leq t\leq\pi$. 
Therefore, we can find a continuous function $\alpha:\mathbb{R}\rightarrow\mathbb{R}$, with $\alpha(0)=0$, such that $d(Z^t(p),Z^{\alpha(t)}(\varphi^{-1}_p(v)))<\delta$, for every $t\in\mathbb{R}$. 
Now, since $Z\in\mathcal{E}^1_{\mu}(M)$, $\varphi^{-1}_p(v)=Z^s(p)$, for $\left|s\right|\leq\epsilon$. 
This is a contradiction, because $\varphi^{-1}_p(v)\in\varphi^{-1}_p(N_p)\backslash\{p\}$.

Now, if $\left|\lambda\right|=1$ but $\lambda\neq1$, we point out that, by \cite[Lemma 3.2]{BR}, we can find $W\in\mathcal{U}$ such that $P^{\pi}_W(p)$ is a rational rotation. Then, there is $T\neq0$ such that $P^{T+\pi}_W(p)=id$. So, we can go on with the previous argument in order to reach the same contradiction.
So, any closed orbit of $X$ is hyperbolic.\end{proof}

\begin{lemma}\label{mainlemma2}
If $X\in int\;\mathcal{S}_{\mu}^1(M)$ then any closed orbit of $X$ is hyperbolic.
\end{lemma}

\begin{proof}
Take $X\in int\;\mathcal{S}_{\mu}^1(M)$, $\mathcal{U}$ a $C^1$-neighbourhood of $X$ in $\mathcal{S}_{\mu}^1(M)$ and 
$p$ be a closed orbit of $X$ with period $\pi>0$.
By contradiction, assume that there is an eigenvalue $\lambda$ of $P_X^{\pi}(p)$ such that $\left|\lambda\right|=1$.

By Zuppa's Theorem (see \cite{Z}), we can find $Y\in\mathcal{U}$ such that \linebreak$Y\in\mathfrak{X}^{\infty}_{\mu}(M)$, $Y^{\pi}(p)=p$ and $P_Y^{\pi}(p)$ has an eigenvalue $\lambda$ with $\left|\lambda\right|=1$, as we remarked in the proof of Lemma \ref{mainlemma}. 

Consider $\varphi$ and $Z\in\mathcal{U}$ as described in the proof of Lemma \ref{mainlemma} and
\begin{equation}
f_Z(x) = \left\{ 
\begin{array}{ll}
\varphi_p^{-1}\circ P^{\pi}_Y(p)\circ\varphi_p(x) & \textrm{, $x\in B_{\epsilon_0}(p)\cap\varphi^{-1}_p(N_{p})$}\nonumber\\
f_Y(x) & \textrm{, $x\notin B_{4\epsilon_0}(p)\cap\varphi^{-1}_p(N_{p})$},
\end{array} 
\right.
\end{equation} where $\epsilon_0>0$ is small.

As it was explained in the proof of Lemma \ref{mainlemma}, we can assume $\lambda=1$ and fix the associated non-zero eigenvector $v$ such that $\|v\|=\epsilon_0/2$. Define $\mathcal{I}_v=\{sv: 0\leq s\leq 1\}$. 

Since $Z\in\mathcal{S}^1_{\mu}(M)$, for any $\epsilon>0$ there is $\delta>0$ such that every $(\delta,T)$-pseudo-orbit is $\epsilon$-shadowed by some orbit $y$ of $Z^t$, for $T>0$. 
Fix $0<\epsilon<\dfrac{\epsilon_0}{4}$.
The idea now is to construct a $(\delta,T)$-pseudo-orbit of $Z^t$, adapting the strategy described on \cite[Proposition A]{LS}. 
Let us present the highlights of that proof.

Let $x_0=p$ and $t_0=0$. Since $p$ is a parabolic closed orbit, we construct a finite sequence $\{(x_i,t_i)\}_{i=0}^I$, where $I\in \mathbb{N}$, $t_i>0$,  $x_i\in\varphi_p^{-1}(\mathcal{I}_v)$, for $1\leq i\leq I$, such that: $x_I=\varphi_p^{-1}(v)$;  $d(Z^t(f_Z(x_i)),Z^t(x_{i+1}))<\delta$, for $\left|t\right|\leq T$ and $0\leq i\leq I-1$; $Z^{t_i}(x_i)=f_Z(x_i)$. 
So, letting $S_n=\sum_{i=0}^nt_i$, for $0\leq n\leq I$, the map $\psi:\mathbb{R}\rightarrow M$ defined by

\begin{equation}
\psi(t) = \left\{ 
\begin{array}{ll}
Z^t(x_0) & \textrm{, $t<0$}\nonumber\\
Z^{t-S_n}(x_{n+1}) & \textrm{, $S_n\leq t< S_{n+1}$, $0\leq n\leq I-2$}\nonumber\\
Z^{t-S_{I-1}}(x_{I}) & \textrm{, $t\geq S_{I-1}$}.
\end{array} 
\right.
\end{equation}
is a $(\delta,T)$-pseudo-orbit of $Z^t$. Now, since $Z\in\mathcal{U}$, there is a reparame\-tri\-zation $\alpha$ and a point $y\in B_{\epsilon}(p)\cap \varphi^{-1}_p(N_{p,\epsilon})$ which $\epsilon$-shadows $\psi$, that is, $d(Z^{\alpha(t)}(y),\psi(t))<\epsilon$, for any $t\in\mathbb{R}$.
Note that, since $\lambda=1$, $$d(x_0,x_I)=d(p,\varphi^{-1}_p(v)=d(p,f_Z(\varphi^{-1}_p(v)))=\left\|v\right\|=\dfrac{\epsilon_0}{2}>2\epsilon.$$

But, since $Z$ has the shadowing property, $$d(x_0,x_I)\leq d(x_0,Z^{\alpha(S_{I-1})}(y))+ d(Z^{\alpha(S_{I-1})}(y),\psi(S_{I-1}))<2\epsilon,$$ which is a contradiction.
\end{proof}

\begin{lemma}\label{linsinglemma}
If $X\in\mathfrak{X}^1_{\mu}(M)$ has a singularity then, for any neighbourhood $\mathcal{V}$ of $X$, there is an open and nonempty set $\mathcal{U}\subset\mathcal{V}$ such that any $Y\in\mathcal{U}$ has a linear hyperbolic singularity.
\end{lemma}

\begin{proof}
Let $p$ be a singularity of $X\in\mathfrak{X}^1_{\mu}(M)$ and $\epsilon>0$.
By a small $C^1$-conservative perturbati\-on of $X$ (see \cite{BR}), we can find $X_1$, $\epsilon$-$C^1$-close to $X$, with a hyperbolic singularity $p$. 
Denote by $\mathcal{V}$ a $C^1$-neighbourhood of $X_1$ in $\mathfrak{X}^1_{\mu}(M)$ where the analytic continuation of $p$ is well-defined.
Now, by Zuppa's Theorem (see \cite{Z}), there is a smooth vector field $X_2\in\mathcal{V}$ with a hyperbolic singularity $p_2$. 
If the eigenvalues of $DX_2(p_2)$ satisfy the nonresonance conditions of the Sternberg linearization theorem (see \cite{SS}) then there is a smooth diffeomorphism conjugating $X_2$ and its linear part around $p_2$. 
If the nonresonance conditions are not satisfied then we can perform a $C^1$-conservative perturbation of $X_2$, so that the eigenvalues satisfy the nonresonance conditions.
So, since the set of divergence-free vector fields satisfying the nonresonance conditions is an open and dense set in $\mathfrak{X}^1_{\mu}(M)$, there is a $C^1$-neighbourhood ${\mathcal{U}}$ of $X_2$ in $\mathcal{V}$ such that any vector field $X_3\in{\mathcal{U}}$ is conjugated to its linear part, meaning that $X_3$ has a linear hyperbolic singularity.
\end{proof}

\begin{proof}[Proof of Theorem \ref{mainth1}]
Take $X\in int \:\mathcal{E}_{\mu}^1(M)$ and let $\mathcal{U}$ be a $C^1$-neighbour\-hood of $X$ in $\mathcal{E}_{\mu}^1(M)$, small enough such that Theorem \ref{BGV} holds. 

Recall that a conservative version of Pugh and Robinson's \textit{Gen\-er\-al Density Theorem} (see \cite{PR}) asserts that, $C^1$-generically, the closed orbits are dense in $M$. 
Denote by $\mathcal{PR}^1_{\mu}(M)$ the Pugh and Robinson's residual set in $\mathfrak{X}^1_{\mu}(M)$ and by $\mathcal{R}$ the residual set given by Theorem \ref{topmix}.

By contradiction, assume that there is $p \in Sing(X)$.
By Lemma\:\ref{linsinglemma}, there is $Y\in\mathcal{U}\cap\mathcal{R}\cap\mathcal{PR}^1_{\mu}(M)$ such that $p\in Sing(Y)$ is linear hyperbolic, and so of saddle-type.
So, by Proposition \ref{vivier}, ${P}_Y^t$ does not admit any dominated splitting over $M\backslash Sing(Y)$. 

We point out that, by Lemma \ref{mainlemma}, any closed orbit of $Y$ is hyperbolic. 
Now, as in the proof of \cite[Lemma 3.1]{F}, take a closed orbit $x$ of $Y$ with arbitrarily large period. So, by Theorem~\ref{BGV}, there are constants $\ell,\tau>0$ such that ${P}_Y^t$ admits an $\ell$-dominated splitting over the $Y^t$-orbit of $x$ with period $\pi(x)>\tau$.
Since $Y\in\mathcal{R}$, by the volume preserving Arnaud Closing Lemma (see \cite[p.13]{Ar}), there is a sequence of vector fields $Y_n\in\mathcal{U}\cap\mathcal{R}$, $C^1$-converging to $Y$, and, for every $n\in\mathbb{N}$, $Y_n$ has a closed orbit $\gamma_n=\gamma_n(t)$ of period $\pi_n$ such that $\displaystyle\lim_{n\rightarrow\infty}\gamma_n(0)=x$ and $\displaystyle\lim_{n\rightarrow\infty}\pi_n=+\infty$. 
Therefore, by Theorem~\ref{BGV}, ${P}_{Y_n}^t$ admits an $\ell$-dominated splitting over the orbit $\gamma_n$, for large $n$. 
Choosing $i\in J\subseteq\mathbb{N}$, there is a sequence of $Y_i$ with ${P}_{Y_i}^t$ having an $\ell$-dominated splitting on a closed orbit $p_i$ and such that the dimensions of the invariant bundles do not depend on $i$. 
Then, given that $M=\displaystyle\limsup_{n}\gamma_n=\displaystyle\bigcap_{N\in\mathbb{N}}\biggl(\overline{\bigcup_{n\geq N}^{\infty}\gamma_n}\biggr),$ we prove that ${P}_Y^t$ admits a dominated splitting over $M\backslash Sing(Y)$. 
But this is a contradiction.
So, $Sing(X)=\emptyset$ and, by Lemma \ref{mainlemma}, one has that if $X\in int\:\mathcal{E}_{\mu}^1(M)$ then $X\in\mathcal{G}^1_{\mu}(M)$.
Then, by Theorem \ref{mainth}, $X$ is Anosov.

Now, take $X\in int\:\mathcal{S}_{\mu}^1(M)$. Applying Lemma \ref{mainlemma2}, we can follow an analogous strategy to that one described above and prove that if $X\in int\:\mathcal{S}_{\mu}^1(M)$ then $Sing(X)=\emptyset$ and $X$ is Anosov.

In order to conclude the proof of Theorem \ref{mainth1}, it is enough to see that $\mathcal{LS}_{\mu}^1(M)\subset\mathcal{S}_{\mu}^1(M)$ and that $\mathcal{A}_{\mu}^1(M)\subset\mathcal{LS}_{\mu}^1(M)$, by \cite[Theorem\:1.5.1]{P}.
\end{proof}

\end{section}


\section*{Acknowledgements}
I would like to thank my supervisors, M\'ario Bessa and Jorge Rocha, whose suggestions and guidance enabled me to develop this work. 

The author was supported by Funda\c c\~ao para a Ci\^encia e a Tecnologia, SFRH/BD/33100/2007.


\end{document}